\documentclass{birkjour}
 \newtheorem{thm}{Theorem}[section]
 
 \newtheorem{lem}[thm]{Lemma}
 \newtheorem{prop}[thm]{Proposition}
 \theoremstyle{definition}
 
 \theoremstyle{remark}

 \numberwithin{equation}{section}
 \usepackage{comment}

\begin{document}

%
%
%
%
%
%
%
%
%

\title[]
 {Weighted composition semigroups on Banach spaces of holomorphic functions}

\author[Bernard]{E. Bernard}

\address{Universit\'e Gustave Eiffel, \\
 LAMA, (UMR 8050), UPEM, UPEC, \\
  CNRS, F-77454, Marne-la-Vall\'ee \\
  France}

\email{eddy.bernard@univ-eiffel.fr}

\thanks{}



\date{September 21, 2021}
\dedicatory{}

\begin{abstract}
We study, on certain Banach spaces $X$, families of weighted composition operators. Notably, we show that if this family forms a strongly continuous semigroup, then its infinitesimal generator ($\Gamma, D(\Gamma)$) is given by $\Gamma f = gf+Gf^\prime$ with $D(\Gamma) = \{ f\in X \ | \ gf+Gf^\prime\in X \}$ where $g, G$ are holomorphic functions. Moreover, our second main result is to study the reciprocal implication. That is, if $(\Gamma , D(\Gamma))$, defined as above, generates a strongly continuous semigroup, then this one is a family of weighted composition operators.
\end{abstract}

\maketitle
\section{Introduction}

Let $\Omega$ be a non-empty open simply connected subset of $\mathbb{C}$ and let $X$ be a complex Banach space which embeds continuously in the Fr\'echet space of all holomorphic functions on $\Omega$, denoted by Hol($\Omega$).
Let $G\in$ Hol($\Omega$). In order to define a semiflow, we consider the following initial value problem:
\begin{equation}
\label{eqCauchy}
u^\prime(t)=G(u(t)) \ \text{ with } \ u(0)=z\in \Omega.
\end{equation}
By 
\begin{equation*}
\varphi(.,z) : [0, \tau(z)) \rightarrow \Omega,
\end{equation*}
we denote the unique maximal solution of \eqref{eqCauchy}. \\ We denote by $\Phi = (\varphi(t, .))_{0\leq t < \tau(.)}$ the semiflow generated by $G$. For $z\in \Omega$, we have
\begin{equation}
\label{semigroupP}
\varphi(t+s, z) = \varphi(t, \varphi(s, z))
\end{equation}
where $t+s<\tau(z)$ and $t<\tau(\varphi(s, z))$. We call the semiflow global if $\tau(z)= \infty$ for all $z\in \Omega$. In that case \eqref{semigroupP} holds for all $t, s \geq 0$ and $z\in \Omega$. Moreover, according to Theorem 2 in \cite[Chap.8, §7]{Hirsch}, $\varphi: [0, \infty) \times \Omega \rightarrow \Omega$ is continuous and according to Theorem 9 in \cite[Chap.1]{Hurewicz}, $\varphi(t, .) \in$  Hol($\Omega$) for all $t\geq 0$. Conversely, let $(\varphi_t)_{t\geq 0}$ be a family of analytic functions on $\Omega$ such that the map $t\mapsto \varphi_t(z)$ is continuous for every $z\in \Omega$, $\varphi_0(z)=z$ and $\varphi_{t+s}(z)=\varphi_t(\varphi_s(z))$ for all $z\in \Omega$ and $t, s\in [0, \infty)$. Berkson and Porta showed in \cite{BP} that there exists a holomorphic map $G: \Omega \rightarrow \mathbb{C}$ such that, for all $t\geq 0$, we have
\begin{equation*}
	\frac{d}{dt} \varphi_t(z)=G(\varphi_t(z)) \ \ \ \ (z\in \Omega),
\end{equation*}
and $(\varphi_t)_{t\geq 0}$ is the solution of the following Cauchy problem:
\begin{equation*}
	\left\lbrace \begin{aligned}
	  &\frac{d}{dt} \varphi_t(z)=G(\varphi_t(z)), &  (t\in \mathbb{R}_+, z\in \Omega)  \\
	& \varphi_0(z)=z, &  (z\in \Omega).
	 \end{aligned}\right. 
\end{equation*}	
Then, the family $(\varphi_t)_{t\geq 0}$ is a global semiflow generated by $G$. Furthermore, 
\begin{equation*}
G(z) = \underset{t\rightarrow 0^+}{\lim} \frac{\varphi_t(z)-z}{t}.
\end{equation*}

We now consider a family of weighted composition operators $(S_t)_{t\geq 0}$ on $X$ defined by:
\begin{equation}
\label{1}
S_tf = m_tf\circ \varphi_t,
\end{equation}
for all $t\geq 0$ and $f\in X$, where $(m_t)_{t\geq 0}$ is a family of holomorphic functions on $\Omega$. Several results are known on this kind of operators where the Banach space $X$ is the Hardy space $H^p(\mathbb{D})$ ($1\leq p < \infty$). See for example \cite{Siskakis86} and \cite{Konig}. Our first goal will be to generalize their results with a Banach space $X$ satisfying some conditions. In the next section, we will show that if $(S_t)_{t\geq 0}$ is a strongly continuous semigroup on $X$, then $(m_t)_{t\geq 0}$ is a differentiable cocycle for $\Phi$, that is \\

i) $m_0(z)=1$ for every $z\in \Omega$;\\
\indent
ii) $m_{t+s}(z)=m_t(z)m_s(\varphi_t(z))$ for all $z\in \Omega$ and $t,s\geq 0$;\\
\indent
iii) the map $t\mapsto m_t(z)$ is continuous for every $z\in \Omega$; \\
\indent
iv) the map $t\mapsto m_t(z)$ is differentiable for all $z\in \Omega$. \\
\\
\noindent
Note that K\"onig showed that a cocycle has no zero (see \cite[Lemma~2.1]{Konig}). Moreover, we also see that the infinitesimal generator ($\Gamma, D(\Gamma)$) of $(S_t)_{t\geq 0}$ is given by 
\begin{equation}
\label{2}
	\left\lbrace \begin{aligned}
	& \Gamma f = gf+Gf^\prime \\
	 & D(\Gamma) = \{ f\in X \ | \ gf+Gf^\prime\in X \},
	\end{aligned} \right.
	\end{equation}
where $g: \Omega \rightarrow \mathbb{C}$ is a holomorphic function. In this paper, basic results on strongly continuous semigroup will be used. For more details one can consult \cite{EngelNagel} or \cite{Pazy}. \\
\\
Then, we will study the converse, that is if the infinitesimal generator ($\Gamma, D(\Gamma)$) of $(S_t)_{t\geq 0}$ is of the form \eqref{2}, then the strongly continuous semigroup has the form \eqref{1}. Recently, two papers \cite{Gutierrez} and \cite{Chalendar} treat in parallel the case where $g=0$ for different conditions on $X$. They show that the strongly continuous semigroup has the form \eqref{1}, but with the weight $m_t=1$ for every $t\geq 0$. In this paper, we generalize this result for $g\neq 0$ with the same conditions used in \cite{Chalendar}. Moreover we comment on the recent generalization obtained in \cite{Gutierrez21}.\\
\\
\noindent
In the following sections we use the maps $e_n: \Omega \rightarrow \mathbb{C}$ ($n\in \mathbb{N}_0=\mathbb{N}\cup \{0\}$) defined by $e_n(z)=z^n$ for all $z\in \Omega$.

\section{Generator of weighted composition semigroups}

First of all, we show that if $\Omega$ is bounded, the maps $e_0, e_1 \in X$ and $(S_t)_{t\geq 0}$ is a strongly continuous semigroup on $X$, then 
\begin{equation}
\label{3}
	m_t(z)=\exp\left( \int_0^t g(\varphi_s(z))ds \right),
\end{equation}
where $g:=\left. \frac{\partial}{\partial t}m_t(z)\right|_{t=0}$ is analytic. The proof of \cite[Theorem~1]{Konig} for the Hardy spaces $H^p(\mathbb{D})$ $(1\leq p < \infty)$ is also valid in our more general situation. \\

\begin{prop}
\label{p2-1}
 Let $\Omega$ be a non-empty open and simply connected subset of $\mathbb{C}$ and let $X$ be a complex Banach space which embeds continuously in the Fr\'echet space of all holomorphic functions on $\Omega$. Furthermore, we suppose that $e_0, e_1 \in X$. If $(S_t)_{t\geq 0}$ defined by \eqref{1} is a strongly continuous semigroup on $X$, then $(m_t)_{t\geq 0}$ is a differentiable cocycle for $\Phi$, the map $g:=\left. \frac{\partial}{\partial t}m_t(z)\right|_{t=0}$ is analytic on $\Omega$ and 
	\begin{equation*}
	m_t(z)=\exp\left( \int_0^t g(\varphi_s(z))ds \right).
	\end{equation*} 
\end{prop}

\leavevmode \\
	Now, we give a condition on $g$ such that $(S_t)_{t\geq 0}$ defined by \eqref{1} and \eqref{3} is a strongly continuous semigroup on $X$. The following theorem generalizes partially Theorem 2 of \cite{Konig} and Theorem 1 of \cite{Siskakis86} in the framework where $X$ is a uniformly convex Banach space. In their papers, only the Hardy spaces $H^p(\mathbb{D})$ ($1\leq p<\infty$) have been considered. Note that, for all $1<p<\infty$, the Hardy spaces $H^p(\mathbb{D})$ and the Bergman spaces $A^p(\mathbb{D})$ are uniformly convex Banach spaces and satisfy the conditions of the following theorem.  
	
\begin{thm}
\label{t2-2}
	Let $\Omega$ be a non-empty open and simply connected subset of $\mathbb{C}$ and let $X$ be a uniformly convex Banach space which embeds continuously in the Fr\'echet space of all holomorphic functions on $\Omega$. Let $g: \Omega \rightarrow \mathbb{C}$ be an analytic function satisfying $\underset{z\in \Omega}{\sup} \ \Re[g(z)] <\infty$. Let $(S_t)_{t\geq 0}$ defined by \eqref{1} and \eqref{3}. Suppose that
	\begin{equation*}
	\left\|hf \right\|_X \leq \left\| h \right\|_\infty \left\| f\right\|_X \ \text{ for every } f\in X \text{ and } h\in H^\infty
	\end{equation*}
	 and
	\begin{equation*}
	\left\|f \circ \varphi_t \right\|_X \leq C_t \left\|  f\right\|_X \ \text{ for all } t\geq 0 \text{ and } f\in X,
	\end{equation*}
     where $C_t$ is a continuous function with respect to $t$, with positive real values and satisfying $\underset{t\rightarrow 0^+}{\lim} C_t =1$. Then, $(S_t)_{t\geq 0}$ is a strongly continuous semigroup on $X$.
\end{thm}

\begin{proof} By definition of $(S_t)_{t\geq 0}$, we have $S_0=Id_X$ and $S_{t+s}=S_t \circ S_s$ for all $s, t \in [0, \infty)$. Moreover, since $\underset{z\in \Omega}{\sup} \ \Re[g(z)] < M$ for some $M\in \mathbb{R}$, we have
	
	\begin{align*}
	\left\| m_t \right\|_\infty &= \left\| \exp\left( \int_0^t \Re g(\varphi_s(z))ds \right) \right\|_\infty \\
	& =\underset{z\in \Omega}{\sup} \exp\left( \int_0^t \Re g(\varphi_s(z))ds \right) \leq e^{Mt}.
	\end{align*}
	
\noindent Hence, $\underset{t\rightarrow 0^+}{\limsup} \left\| m_t \right\|_\infty \leq 1$ and by a generalization of \cite[Lemma~2.1]{Konig} to $\Omega$, open and simply connected, we get that $m_t$ is bounded for all $t\geq 0$. Let $f\in X$. By hypothesis, we have
	\begin{equation}
	\label{star}
	\left\|S_tf \right\|_X \leq \left\| m_t \right\|_\infty \left\| f \circ \varphi_t\right\|_X \leq C_t \left\| m_t \right\|_\infty \left\| f \right\|_X < \infty. 
	\end{equation} 
	Therefore, $S_t: X\rightarrow X$ is well-defined and $\{\left\| S_t \right\|  \ : \ t\in [0,1]  \}$ is bounded. \\
	\\
	Let $(t_n)_{n\in \mathbb{N}}$ be a sequence in $[0,1]$ which converges to $0$. Then $(S_{t_n}f)_n$ is a bounded sequence in $X$. From Milman-Pettis Theorem, since $X$ is a uniformly convex Banach space, $X$ is a reflexive space. Thus, since each bounded sequence in $(X, \left\|.  \right\|_X )$ possesses a sub-sequence which converges weakly, there exists an increasing map $\psi : \mathbb{N} \rightarrow \mathbb{N}$ such that $S_{t_{\psi(n)}}f \rightharpoonup h \in X$. However, by hypothesis, $X\hookrightarrow Hol(\Omega)$, that is the evaluations $\delta_z: X\rightarrow \mathbb{C}$ which map $f$ to $f(z)$ are continuous linear forms. Then, by definition, we have $S_{t_{\psi(n)}}f(z) \overset{n\rightarrow +\infty}{\longrightarrow} h(z)$ for all $z\in \Omega$. However, $S_{t_{\psi(n)}}f(z) \overset{n\rightarrow +\infty}{\longrightarrow} f(z)$. Therefore, $S_{t_{\psi(n)}}f \rightharpoonup f$. So, $\left( \left\| S_{t_{\psi(n)}}f \right\|_X  \right)_{n\geq 0}$ is bounded and $\left\| f \right\|_X \leq \underset{n\rightarrow +\infty}{\liminf} \left\| S_{t_{\psi(n)}}f \right\|_X$. Furthermore, from Inequality \eqref{star} and the fact that \\
	$\underset{t\rightarrow 0^+}{\limsup} \left\| m_t \right\|_\infty \leq 1$ and $\underset{t\rightarrow 0^+}{\lim} C_t =1$, we have
	\begin{equation*}
	\underset{n\rightarrow +\infty}{\limsup} \left\| S_{t_{\psi(n)}}f \right\|_X \leq \left\| f \right\|_X.
	\end{equation*}
	Hence, $\underset{n\rightarrow +\infty}{\lim} \left\| S_{t_{\psi(n)}}f \right\|_X = \left\| f \right\|_X$. Moreover, since $X$ is a uniformly convex Banach space, $S_{t_{\psi(n)}}f \rightharpoonup f$ and $\underset{n\rightarrow +\infty}{\lim} \left\| S_{t_{\psi(n)}}f \right\|_X = \left\| f \right\|_X$, we have $S_{t_{\psi(n)}}f \overset{X}{\rightarrow} f$. Hence, $\underset{n\rightarrow +\infty}{\lim} \left\|S_{t_{\psi(n)}}f -f \right\|_X = 0$ and thus $(S_t)_{t\geq 0}$ is a strongly continuous semigroup on $X$.
	\end{proof}

\noindent The following theorem describes the generator of a weighted composition semigroup. The proof of \cite[Theorem~2]{Siskakis86} for the Hardy spaces $H^p(\mathbb{D})$ $(1\leq p < \infty)$ is also valid in our more general situation. \\
	
\begin{thm}
\label{t2-3}	
Let $X$ be a complex Banach space which embeds continuously in Hol($\mathbb{D}$). Moreover, we suppose that $e_0, e_1 \in X$. If $(S_t)_{t\geq 0}$ defined by \eqref{1}, is a strongly continuous semigroup on $X$, then, 
	\begin{equation*}
	m_t(z)=\exp\left( \int_0^t g(\varphi_s(z))ds \right)
	\end{equation*}
	with $g: \mathbb{D} \rightarrow \mathbb{C}$ holomorphic and the infinitesimal generator ($\Gamma, D(\Gamma))$ of $(S_t)_{t\geq 0}$ is given by
	\begin{equation*}
	\left\lbrace \begin{aligned}
	&\Gamma f = gf+Gf^\prime \\
	&D(\Gamma) = \{ f\in X \ | \ gf+Gf^\prime\in X \},
	\end{aligned} \right. 
	\end{equation*}
	 where $G$ is the generator of $(\varphi_t)_{t\geq 0}$.
\end{thm}

	Next we give conditions on $X$ so that the generator $\Gamma$ is continuous if and only if $G\equiv 0$. 
	
	\begin{thm}
		\vspace{0.2cm}
		\label{C2-2-1-3a}
		\noindent
		Let $X$ be a complex Banach space which embeds continuously in Hol($\mathbb{D}$). Consider the family $(S_t)_{t\geq 0}$ defined by \eqref{1}, where the semiflow $\Phi$ has $G$ for generator. Suppose that $(S_t)_{t\geq 0}$ is a strongly continuous semigroup on $X$, $e_n \in D(\Gamma)$ for all $n\in \mathbb{N}_0$ where $\Gamma$ is the generator of $(S_t)_{t\geq 0}$. Moreover, we suppose that: \\
		\indent
		a) $\exists C > 0$ such that $\Vert  hf\Vert  _X\leq C\Vert  h\Vert  _\infty . \Vert  f\Vert  _X$ ($h\in H^\infty(\mathbb{D}), f\in X$) \\
		\indent
		b)  $e_nf\in X$ ($f\in X, n\in \mathbb{N}$) \\
		\indent
		c) $\exists C_n > 0$ such that $C_n \Vert  f\Vert  _X \leq \Vert  e_nf\Vert  _X$ ($f\in X, n\in \mathbb{N}$) \\
		\indent
		d) $\underset{n\rightarrow +\infty}{\lim} \frac{\Vert  e_n\Vert  _X}{nC_n} = 0$. \\
		Then,
		\begin{equation*}
		m_t(z)=\exp\left( \int_0^t g(\varphi_s(z))ds \right)
		\end{equation*}
		where $g: \mathbb{D} \rightarrow \mathbb{C}$ is a holomorphic function. If $g$ is bounded, then $\Gamma$ is continuous if and only if $G\equiv 0$. 	
	\end{thm}
	
	\begin{proof}
		\vspace{0.2cm}
		\noindent
		According to Theorem \ref{t2-3}, the infinitesimal generator ($\Gamma, D(\Gamma))$ of $(S_t)_{t\geq 0}$ is given by
		\begin{equation*}
		\left\lbrace \begin{array}{l}
		\Gamma f = gf+Gf^\prime \\
		D(\Gamma) = \{ f\in X \ | \ gf+Gf^\prime\in X \}.
		\end{array} \right. 
		\end{equation*}
		Suppose that $G\equiv 0$. Then, $\Vert  \Gamma f\Vert  _X =\Vert  gf\Vert  _X \leq C \Vert  g\Vert  _\infty . \Vert  f\Vert  _X \leq C_1 \Vert  f\Vert  _X$. Hence, $\Gamma$ is continuous. \\
		\\
		Suppose that $\Gamma$ is continuous. Since $e_0 \in D(\Gamma)$, we have $\Vert  \Gamma e_0\Vert  _X = \Vert  g\Vert  _X \leq \Vert  \Gamma \Vert  .\Vert  e_0\Vert  _X$. Hence, $g\in X$. Since $e_1 \in D(\Gamma)$, we have $\Vert  G\Vert  _X \leq \Vert  \Gamma\Vert  .\Vert  e_1\Vert  _X+\Vert  ge_1\Vert  _X\leq \Vert  \Gamma\Vert  .\Vert  e_1\Vert  _X+\Vert  g\Vert  _\infty.\Vert  e_1\Vert  _X$. Hence, $G\in X$. Since $e_n \in D(\Gamma)$, by our hypotheses, we have 
		$$ C_{n-1}\Vert  ge_1+nG\Vert  _X \leq \Vert  (ge_1+nG)e_{n-1}\Vert  _X = \Vert  ge_n+nGe_{n-1}\Vert  _X\leq \Vert  \Gamma\Vert  .\Vert  e_n\Vert  _X.$$ 
		Thus, 
		$$ \Vert  G\Vert  _X \leq \Vert  \Gamma\Vert  . \frac{\Vert  e_n\Vert  _X}{nC_{n-1}}+ \frac{\Vert  g\Vert  _\infty.\Vert  e_1\Vert  _X}{n}.$$
		Taking the limit $n\rightarrow +\infty$, we get that $G\equiv 0$.
		\vspace{0.2cm}
	\end{proof}
	
	For instance, we can consider the Hardy space $H^2(\mathbb{D})$. Since $\Vert  e_nf\Vert  _2 = \Vert  f\Vert  _2$ and $\Vert  e_n\Vert  _2=1$ for all $n\geq 0$, the space $H^2(\mathbb{D})$ sastifies all the conditions of Theorem \ref{C2-2-1-3a}. We can also consider weighted Dirichlet spaces $D_\alpha$ with $\alpha \in \mathbb{R}_+$. Recall that $D_\alpha$
	is the space of analytic function $f$ on $\mathbb{D}$ such that
	$$ \int_\mathbb{D} |f^\prime(z)|^2 (1-|z|^2)^\alpha \frac{dA(z)}{\pi} < \infty.$$
	First of all, we recall the following result.
	
	\begin{prop}
	\vspace{0.2cm}
	\label{C2-2-1-3n2}
	\noindent
Let $\alpha \in \mathbb{R}_+$. Let $w_n := 2n^2 \int_0^1 r^{2n-1}w(r)dr = n \frac{n!}{(\alpha + 1)...(\alpha+n)}$ for all $n\geq 1$. If $f=\underset{n\geq 0}{\sum} a_n z^n \in D_\alpha$, then
	$$\Vert f\Vert^2_{D_\alpha} = \underset{n\geq 0}{\sum} w_n |a_n|^2.$$
	Moreover, $D_\alpha$ embeds continuously in $Hol(\mathbb{D})$.
\end{prop}

Next, we prove the following proposition.
	
	\begin{prop}
		\vspace{0.2cm}
		\label{C2-2-1-3n7}
		\noindent
		Let $\alpha \in \mathbb{R}_+$. For all $f\in D_\alpha$ and $k\in \mathbb{N}$, $e_kf\in D_\alpha$ and there exists $C>0$ such that $C \Vert  f\Vert  _{D_\alpha} \leq \Vert  e_kf\Vert  _{D_\alpha}$ satisfying $\underset{k\rightarrow +\infty}{\lim} \frac{\Vert  e_k\Vert  _{D_\alpha}}{kC} = 0$.
	\end{prop}
	
	\begin{proof}
		\vspace{0.2cm}
		\noindent
		Let $f=\underset{n\geq 0}{\sum} a_n z^n\in D_\alpha$. We have 
		$$\Vert  e_k f\Vert  ^2_{D_\alpha} = \underset{n\geq 0}{\sum} w_{n+k} |a_n|^2.$$
		Hence,
		$$ \frac{\Vert  e_k f\Vert  ^2_{D_\alpha}}{\Vert  f\Vert  ^2_{D_\alpha}} = \underset{n\geq 0}{\sum} \frac{w_{n+k}}{w_n}.$$
		Since $\frac{w_{n+k}}{w_n} \rightarrow 1$ where $n\rightarrow +\infty$, the series is divergent and so there exists $C>0$ such that $C \Vert  f\Vert  _{D_\alpha} \leq \Vert  e_kf\Vert  _{D_\alpha}$. On the other hand, for all $k\geq 1$, by Stirling formula, we have  
		$$\Vert  e_k\Vert  ^2_{D_\alpha} = w_k = k\frac{\alpha!k!}{(\alpha+k)!} \underset{+\infty}{\sim} \frac{k^{k+3/2}}{k^{\alpha+k+1/2}} = k^{1-\alpha}.$$
		Then, 
		$$\frac{\Vert  e_k\Vert  _{D_\alpha}}{kC} \underset{+\infty}{\sim} \frac{k^{1/2(1-\alpha)}}{kC} = \frac{1}{k^{1/2(1+\alpha)}} \rightarrow 0.$$
		Therefore,  $\underset{k\rightarrow +\infty}{\lim} \frac{\Vert  e_k\Vert  _{D_\alpha}}{kC} = 0$.
		\vspace{0.2cm}
	\end{proof}

	It is well known that the other conditions of Theorem \ref{C2-2-1-3a} are satisfied by $D_\alpha$ when $\alpha \in [1, +\infty)$, so we can apply it to these spaces. \\
	\\
 We now study the reciprocal implication, that is if the infinitesimal generator ($\Gamma, D(\Gamma)$) of $(S_t)_{t\geq 0}$ is of the form \eqref{2}, then the strongly continuous semigroup has the form \eqref{1}. First of all, we consider the same conditions on the Banach space $X$ as in the paper \cite{Chalendar}, that is a evaluation condition of the form: \\
	\\
	\label{E}
	(E)\  \  Let $z_n\in \Omega$ such that $z_n\rightarrow z \in \overline{\Omega}\cup \{\infty\}$ and $\underset{n\rightarrow \infty}{\lim} f(z_n)$ exists in $\mathbb{C}$ for all $f\in X$, then $z\in \Omega$. \\
	\\	
Before to show the reciprocal implication, let us begin by several lemmas. \\

\begin{lem}
\label{l3-1}
		Let $\Omega$ be a non-empty open and simply connected subset of $\mathbb{C}$, $X$ a Banach space of holomorphic functions and $g, G\in $Hol($\Omega$). Let $\Gamma$ be the operator defined by $\Gamma f=gf+Gf^\prime$ on a domain $D(\Gamma)\subset X$. Let $\delta_z$ be the evaluation in $z$ and $\Gamma^\prime$ the adjoint operator of $\Gamma$. Then for all $z\in \Omega$, $\delta_z\in D(\Gamma^\prime)$ and $\Gamma^\prime\delta_z=g(z)\delta_z+G(z)\delta_z^\prime$, where $<\delta_z^\prime, f>=f^\prime(z)$.
	\end{lem}
	
\begin{proof}
	Let $f\in D(\Gamma)$. We have
	\begin{align*}	
	<\Gamma^\prime\delta_z, f>=<\delta_z, \Gamma f> & =<\delta_z, gf+Gf^\prime> \\
	& =g(z)f(z)+G(z)f^\prime(z) \\
	& =g(z)<\delta_z, f>+<G(z)\delta_z^\prime, f> \\
	& = <g(z)\delta_z+G(z)\delta_z^\prime, f>.
	\end{align*}
	Hence $\delta_z\in D(\Gamma^\prime)$ and $\Gamma^\prime\delta_z=g(z)\delta_z+G(z)\delta_z^\prime$.
\end{proof}

\begin{lem}
\label{l3-2}
		Let $\Omega$ be a non-empty open and simply connected subset of $\mathbb{C}$, $X$ a Banach space of holomorphic functions, $g, G\in $Hol($\Omega$) and $(S_t)_{t\geq 0}$ a strongly continuous semigroup on $X$ with generator $(\Gamma, D(\Gamma))$ where $\Gamma f=gf+Gf^\prime$ for all $f\in D(\Gamma)$. Let $f\in X$. Define $u(t,z):=(S_tf)(z)$. Then $\frac{d}{dt} u(t,z)= g(z)u(t,z)+G(z)\frac{d}{dz} u(t,z)$ for all $t\geq 0$ and $z\in \Omega$.
\end{lem}

\begin{proof}
	From the previous lemma, we have
	\begin{align*}
	\frac{d}{dt} u(t,z)=\frac{d}{dt}<\delta_z, S_tf>=<\delta_z, \frac{d}{dt}S_tf> & = <\delta_z, \Gamma S_tf> \\
	& = <\Gamma^\prime\delta_z, S_tf> \\
	& =<g(z)\delta_z+G(z)\delta_z^\prime, S_tf> \\
	& = g(z)S_tf(z) + G(z)(S_tf)^\prime(z) \\ 
	& = g(z)u(t,z)+G(z)\frac{d}{dz} u(t,z).
	\end{align*}
\end{proof}

\begin{lem}
\label{l3-3}
		Let $\Omega$ be a non-empty open and simply connected subset of $\mathbb{C}$, $X$ a Banach space of holomorphic functions. Let $g \in Hol(\Omega)$ and $G\in Hol(\Omega)$ which generates the semiflow $(\varphi_t)_{0\leq t <\tau(z)}$. Let $(S_t)_{t\geq 0}$ be a strongly continuous semigroup on $X$ with generator $(\Gamma, D(\Gamma))$ where $\Gamma f=gf+Gf^\prime$ for all $f\in D(\Gamma)$. Let $f\in X$. Then, $(S_tf)(z)=m_t(z)f(\varphi_t(z))$ for all $z\in \Omega$ and $0\leq t<\tau(z)$, where 
	\begin{equation*}
	m_t(z)=\exp\left( \int_0^t g(\varphi_s(z))ds \right) , \ \ z\in \Omega.
	\end{equation*}
\end{lem}

\begin{proof}
	Define $u(t,z):=(S_tf)(z)$ for all $z\in \Omega$ and $0\leq t<\tau(z)$ and $w:[0,t]\rightarrow \mathbb{C}$ defined by $w(s)=m_s(z)u(t-s,\varphi_s(z))$ for $s\in [0,t]$. From the previous lemma, we have
	\begin{equation*}	
	\begin{split}
	& w^\prime(s)=\frac{d}{ds} w(s)  =\frac{d}{ds} m_s(z) u(t-s,\varphi_s(z)) \\
	  &= \frac{d}{ds} [m_s(z)].u(t-s,\varphi_s(z)) + m_s(z).\frac{d}{ds} [u(t-s,\varphi_s(z))] \\
	 &= m_s(z)g(\varphi_s(z))u(t-s,\varphi_s(z))  \\
	 &+  m_s(z). \left[ -\frac{d}{dt}u(t-s,\varphi_s(z)) + \frac{d}{ds} \varphi_s(z) . \frac{d}{dz} u(t-s,\varphi_s(z)) \right] \\
	 & = m_s(z). \left[ g(\varphi_s(z))u(t-s,\varphi_s(z)) + G(\varphi_s(z))\frac{d}{dz} u(t-s,\varphi_s(z)) \right] \\ 
	 &-m_s(z) \frac{d}{dt}u(t-s,\varphi_s(z)) \\
	 & = m_s(z) \frac{d}{dt}u(t-s,\varphi_s(z)) - m_s(z) \frac{d}{dt}u(t-s,\varphi_s(z)) = 0.
\end{split}	
	\end{equation*}
	Hence $w(t)=w(0)$ with $w(0)=u(t,z)=(S_tf)(z)$ and
	\begin{equation*}
	w(t)=m_t(z)u(0,\varphi_t(z))=m_t(z)(S_0f)(\varphi(t,z))=m_t(z)f(\varphi_t(z)). 
	\end{equation*}
	Therefore, $(S_tf)(z)=m_t(z)f(\varphi_t(z))$ for all $z\in \Omega$ and $0\leq t<\tau(z)$.
\end{proof}	

\begin{lem}
\label{l3-4}
	Let $\Omega$ be a non-empty open and simply connected subset of $\mathbb{C}$ and let $X$ be a complex Banach space which embeds continuously in Hol($\Omega$). Moreover, we suppose that $X$ satisfies the evaluation condition (E). Let $G\in Hol(\Omega)$ generate the semiflow $(\varphi_t)_{0\leq t <\tau(z)}$ and $g\in Hol(\Omega)$. Let $(S_t)_{t\geq 0}$ be a strongly continuous semigroup on $X$ of generator $(\Gamma, \mathcal{D})$, with $\Gamma f=gf+Gf^\prime$ for all $f\in \mathcal{D}$. Then, $\tau(z)=\infty$ for all $z\in \Omega$ and therefore the semiflow $\Phi$ is global on $\Omega$. 
\end{lem}
	
\begin{proof}
	Suppose there exists $z\in \Omega$ such that $\tau(z)<\infty$. Then, there exists an increasing sequence $(t_n)_n$ which converges to $\tau(z)$ and there exists $\alpha\in \partial \Omega\cup \{\infty\}$ such that $\underset{n\rightarrow \infty}{\lim} \varphi(t_n,z)=\alpha$. Let $f\in X$. Then, by continuity and Lemma \ref{l3-3}, we have
	\begin{equation*}
	\underset{n\rightarrow \infty}{\lim} m_{t_n}(z)f(\varphi(t_n,z)) = \underset{n\rightarrow \infty}{\lim} (S_{t_n}f)(z)=(S_{\tau(z)}f)(z).
	\end{equation*}
	Hence,
	\begin{equation*}
	\underset{n\rightarrow \infty}{\lim} f(\varphi(t_n,z)) = \frac{(S_{\tau(z)}f)(z)}{m_{\tau(z)}(z)}
	\end{equation*}
	exists in $\mathbb{C}$ for all $f\in X$. However, from the evaluation condition (E), we have $\alpha\in \Omega$ which is a contradiction with $\alpha\in \partial \Omega\cup \{\infty\}$.
\end{proof}

\begin{thm}
\label{t3-5}
	Let $\Omega$ be a non-empty open and simply connected subset of $\mathbb{C}$ and let $X$ be a complex Banach space which embeds continuously in Hol($\Omega$). Moreover, we suppose that $X$ satisfies the evaluation condition (E). Let $G\in Hol(\Omega)$ generate the semiflow $(\varphi_t)_{0\leq t <\tau(z)}$ and $g\in Hol(\Omega)$. Let $(S_t)_{t\geq 0}$ be a strongly continuous semigroup on $X$ of generator $(\Gamma, \mathcal{D})$ with $\Gamma f=gf+Gf^\prime$ for all $f\in \mathcal{D}$. Then, the semiflow is global and $(S_tf)(z)=m_t(z)f(\varphi_t(z))$ for all $f\in X$, $t\geq 0$ and $z\in\Omega$ with 
\begin{equation*}
m_t(z)=\exp\left( \int_0^t g(\varphi_s(z))ds \right).
\end{equation*} 
Moreover, $\mathcal{D}=\{f\in X \ | \ gf + Gf^\prime \in X\}=:D(\Gamma)$.
\end{thm}

\begin{proof}
 Let $f\in X$. From, Lemma \ref{l3-3} and Lemma \ref{l3-4}, we have $S_tf=m_tf\circ \varphi_t$ for all $t\geq 0$. \\
	\\
	Let $\mathcal{D}$ denote the domain of the generator $\Gamma$. We need to show $\mathcal{D} = D(\Gamma)$. By definition, we have
	\begin{equation*}
\mathcal{D}=\{f\in X \ | \  \underset{t\rightarrow 0^+}{\lim} \frac{S_tf-f}{t} \ \text{ exists in } X\}.
	\end{equation*}
	Let $f\in \mathcal{D}$. Since $X$ embeds continuously in Hol($\Omega$), the convergence in $X$ implies the pointwise convergence. Therefore, for all $z\in \Omega$, we have 
	\begin{align*}
	\underset{t\rightarrow 0^+}{\lim} \frac{S_tf(z)-f(z)}{t} & = \left. \frac{d}{dt} S_tf(z)\right|_{t=0} \\ 
	&= \left. \frac{d}{dt} m_t(z)f(\varphi_t(z))\right|_{t=0} \\
	& = \left[ m_t(z)g(\varphi_t(z))f(\varphi_t(z)) + m_t(z)(Gf^\prime)(\varphi_t(z)) \right]_{t=0} \\
	& =  g(z)f(z)+G(z)f^\prime(z).
	\end{align*}
	Hence, $gf + Gf^\prime \in X$ and $\mathcal{D} \subset D(\Gamma)$. Now we show that the family of operators $(T_t)_{t \geq 0}$ in $X$ defined by $T_tf=f\circ \varphi_t$ is a strongly continuous semigroup whose generator $A$ is of the form $Af=Gf^\prime$. By the semiflow properties, we have $T_0=I_X$ and $T_{t+s}=T_t\circ T_s$ for every $t, s\geq 0$. Moreover, by definition of $T_t$, for all $f\in X$, we have
	\begin{equation*}
	\underset{t\rightarrow 0^+}{\lim} T_tf = \underset{t\rightarrow 0^+}{\lim} \frac{S_tf}{m_t} = \underset{t\rightarrow 0^+}{\lim} S_tf = f.
	\end{equation*}
	Then, the family of operators $(T_t)_{t \geq 0}$ is a strongly continuous semigroup on $X$. Its generator is given by
	\begin{equation*}
	Af = \underset{t\rightarrow 0^+}{\lim} \frac{T_tf-f}{t} = \left. \frac{d}{dt} T_tf\right|_{t=0} =  \left. \frac{d}{dt} \frac{S_tf}{m_t}\right|_{t=0}.
	\end{equation*}
	However,
	\begin{align*}
	\left. \frac{d}{dt} \frac{S_tf(z)}{m_t(z)}\right|_{t=0} & = \left[ m_t(z)\frac{ g(z)S_tf(z)+G(z)(S_tf)^\prime(z)-g(\varphi_t(z))S_tf(z)}{m_t(z)^2} \right]_{t=0} \\
	& = G(z) \left[ (S_tf)^\prime(z) \right]_{t=0} = G(z)f^\prime(z).
	\end{align*}
	Now let us show that $D(\Gamma) \subset \mathcal{D}$. Let $f\in X$ such that $h:=gf+Gf^\prime \in X$. Then, for all $z\in \Omega$, from Lemma \ref{l3-2} applied with $T_sf$ and $T_sId$, we have
	\begin{align*}
	\frac{d}{ds} \left[ m_s(z)f(\varphi_s(z)) \right] & = m_s(z)g(\varphi_s(z))f(\varphi_s(z)) + m_s(z) \frac{d}{ds} T_sf \\
	& = m_s(z)g(\varphi_s(z))f(\varphi_s(z)) + m_s(z)G(z) \frac{d}{dz} f(\varphi_s(z)) \\
	& =  m_s(z)g(\varphi_s(z))f(\varphi_s(z)) + m_s(z) G(z) \frac{d}{dz}\left[ \varphi_s(z) \right] f^\prime(\varphi_s(z)) \\
	& =  m_s(z)g(\varphi_s(z))f(\varphi_s(z)) +  m_s(z) G(\varphi_s(z))f^\prime(\varphi_s(z)) \\
	&= m_s(z)h(\varphi_s(z))
	\end{align*}
	for all $s\geq 0$. \\
	\\
	Let $t>0$. Then, we have 
	\begin{align*}
	\int_{0}^{t} (S_sh)(z)ds = \int_{0}^{t} m_s(z)h(\varphi_s(z)) ds  & = m_t(z)f(\varphi_t(z))-m_0(z)f(\varphi_0(z)) \\
	 & = S_tf(z)-f(z).
	\end{align*}
	Furthermore, by the semigroup properties, we have
	\begin{equation*}
	\underset{t\rightarrow 0^+}{\lim}\frac{1}{t} \int_{0}^{t} (S_sh)(z)ds = S_0h(z)=h(z). 
	\end{equation*}
	Thus, 
	\begin{equation*}
	\underset{t\rightarrow 0^+}{\lim}\frac{S_tf(z)-f(z)}{t} = h(z).
	\end{equation*}
	Hence, $\Gamma f=h$ and $f\in D$. Then, $\mathcal{D} = D(\Gamma)$.
	\end{proof}

\noindent Recently, in \cite{Gutierrez21}, they showed that another condition on Banach spaces $X$ implies the same result. This condition is the following:
$$Hol(\overline{\mathbb{D}}) \hookrightarrow X \hookrightarrow Hol(\mathbb{D})$$
where both embeddings are continuous and $Hol(\overline{\mathbb{D}})$ is the set of $f\in Hol(\mathbb{D})$ such that there exists an open subset $U_f \subset \overline{\mathbb{D}}$ such that $f$ extends analytically on $U_f$. The condition $Hol(\overline{\mathbb{D}}) \hookrightarrow X$ means that for all $\varepsilon > 0$, the disk algebra $A(D(0,1+\varepsilon))$ embeds continuously in $X$, that is $\forall \varepsilon > 0$, $\exists C_\varepsilon > 0$ such that
$$\Vert   f \Vert  _X \leq C_\varepsilon \Vert   f \Vert  _{\infty, D(0,1+\varepsilon)}$$
for all $f\in A(D(0,1+\varepsilon))$. In particular, $e_n: z\mapsto z^n$ are in $A(D(0,1+\varepsilon))$, so $\Vert  e_n\Vert  _X \leq C_\varepsilon (1+\varepsilon)^n$ and thus $\underset{n}{\overline{\lim}} \Vert  e_n\Vert  _X^{\frac{1}{n}}  \leq 1$. In fact, the condition $\underset{n}{\overline{\lim}} \Vert  e_n\Vert  _X^{\frac{1}{n}}  \leq 1$ is equivalent to $Hol(\overline{\mathbb{D}}) \hookrightarrow X$. Indeed, let $\varepsilon > 0$ and $f\in A(D(0,1+\varepsilon))$. Thus, the Taylor coefficients of $f$ are given by Cauchy's formula:
$$a_n = \frac{1}{2\pi i} \int_{|z|=1+\frac{\varepsilon}{2}} \frac{f(z)}{z^{n+1}} dz.$$
Hence, we have
$$|a_n| \leq \frac{2\pi}{2\pi } \left(1+\frac{\varepsilon}{2}\right) \frac{\Vert   f \Vert  _{\infty, D(0,1+\varepsilon)}}{(1+\frac{\varepsilon}{2})^{n+1}} = \frac{\Vert   f \Vert  _{\infty, D(0,1+\varepsilon)}}{(1+\frac{\varepsilon}{2})^{n}}.$$
On the other hand, by the hypothesis, there exists an integer $N$ such that $\Vert e_n \Vert_X^{\frac{1}{n}} \leq \left(1+\frac{\varepsilon}{4}\right)$ for $n\geq N$ and $M>0$ such that  $\Vert e_n \Vert_X^{\frac{1}{n}} \leq M$ for $n< N$. So, by the triangle inequality, we have
$$ \left\| \underset{n=0}{\overset{\infty}{\sum}} a_ne_n \right\|_X \leq \underset{n=0}{\overset{N-1}{\sum}} \frac{M \Vert   f \Vert_{\infty, D(0,1+\varepsilon)}}{(1+\frac{\varepsilon}{2})^{n}} + \underset{n=N}{\overset{\infty}{\sum}} \frac{\Vert   f \Vert_{\infty, D(0,1+\varepsilon)}(1+\frac{\varepsilon}{4})^{n}}{(1+\frac{\varepsilon}{2})^{n}},$$
which shows the existence of the constant $C_\varepsilon$. Thus, we can rewrite \cite[Theorem~3.1]{Gutierrez21} as follows.

\begin{thm}
	\label{t3-6}
Let $X$ be a Banach space of analytic functions on $\mathbb{D}$ satisfying: $e_n\in X$, for all $z\in \mathbb{D}$, evaluation maps $E_z: X\rightarrow \mathbb{C}$ are continuous, and $\underset{n}{\overline{\lim}} \Vert  e_n\Vert  _X^{\frac{1}{n}}  \leq 1$. Let $G\in Hol(\mathbb{D})$ and $g\in Hol(\mathbb{D})$. Define $\Gamma$ by $f\mapsto gf+Gf^\prime$ on $D(\Gamma)=\{f\in X \ | \ gf+Gf^\prime \in X \}$. If $\Gamma$ generates a strongly continuous semigroup on $X$, then it generates a strongly continuous weighted composition semigroup.
\end{thm}

 Both points of view are interesting. Indeed, on the one hand if the norm of $X$ is expressed by means of an integral then the continuous injection will be easier to verify directly and on the other hand, if the norm is expressed with the Taylor coefficients then it will be easier to check $\underset{n}{\overline{\lim}} \Vert  e_n\Vert  _X^{\frac{1}{n}}  \leq 1$. To conclude, note that for the Hardy space $H^2(\mathbb{D})$, for the Bergman spaces $A^p(\mathbb{D}) \ (p\geq 1)$ and the Dirichlet space $\mathcal{D}$, the conditions of Theorem \ref{t3-6} are satisfied.

\subsection*{Acknowledgment}
The author wishes to thank Jonathan Partington for the second part of the proof of the characterization of Banach spaces $X$ satisfying $Hol(\overline{\mathbb{D}}) \hookrightarrow X$.

\end{document}